\documentclass{amsart}
\usepackage{amsmath, amssymb}

\newtheorem{Thm}{Theorem}

\newtheorem{Prop}{Proposition}
\newtheorem{remark}{Remark}

\newcommand{\rmv}[1]{}

\begin{document}
\title{The multisubset sum problem for finite abelian groups}
\author[Muratovi\'{c}-Ribi\'{c}]{Amela Muratovi\'{c}-Ribi\'{c}}
\address{University of Sarajevo, Department of Mathematics, Zmaja od Bosne 33-35, 71000 Sarajevo, Bosnia and Herzegovina}
\email{amela@pmf.unsa.ba}

\author[Wang]{Qiang Wang}
\thanks{Research is partially supported by NSERC of Canada.}
\address{School of Mathematics and Statistics,
Carleton University, 1125 Colonel By Drive,  Ottawa, Ontario, $K1S$
$5B6$, CANADA}

\email{wang@math.carleton.ca}



\begin{abstract}
In this note, we give the explicit formula for the number of
multisubsets of a finite abelian group $G$  with any given size such
that the sum is equal to a given element $g\in G$. This also gives
the number of partitions of $g$ into a given number of parts over a
finite abelian group. An inclusion-exclusion formula for the number
of multisubsets of a subset of $G$  with a given size and a given
sum is also obtained.
\end{abstract}

\maketitle

Let $G$ be a finite abelian group of size $n$ and $D$ be a subset of
$G$.  The well known subset sum problem in combinatorics is to
decide whether there exists a subset $S$ of $D$ which sums to a
given element in $G$. This problem is an important problem in
complexity theory and cryptography and it is NP-complete. For any
$g\in G$ and $i$ a positive integer, we let the number of subsets
$S$  of $D$ of size $i$ which sum up to $g$ be denoted by
\[
N(D, i, g) = \# \{ S \subseteq D : \# S = i, \sum_{s\in S} s = g \}.
\]

The explicit formula of $N(D, i, g)$ in general is a even harder
problem. However, when $D$ has more structure, Li and Wan made some
important progress in obtaining the explicit formula using a sieving
technique \cite{LiWan:08, LiWan:12}. Recently Kosters
\cite{Kosters:13} gives a shorter proof of the explicit formula
using character theory. Namely,

\[
N(G, i, g) = \frac{1}{n} \sum_{s \mid gcd(i, exp(G))} (-1)^{i+i/s}
\binom{n/s}{i/s} \sum_{d  \mid \gcd(e(g), s)} \mu(s/d) \# G[d],
\]

where $exp(G)$ is the exponent of $G$, $e(g) = \max \{ d : d \mid
exp(G), g \in dG \}$, $\mu$ is the M\"{o}bius function, and $G[d] =
\{h \in G: dh =0\}$ is the $d$-torsion of $G$.

Similarly, we denote
\[
M(D, i, g) =  \# \{ M : M ~ is~ a ~ multisubset ~of~ D,  \# M = i,
\sum_{s \in M} s =g \}.
\]
It is an interesting question by its own to count $M(D, i, g)$, the
number of multisubsets of $D$ which sum up to $g$. Indeed, this
problem is also equivalent to count the number of partitions with at
most $i$ parts over $D$. To avoid confusion we denote by
$\{\{a_1,\ldots ,a_n\}\}$ multisets, i.e. with possibly repeated
elements, and by $\{a_1,\ldots ,a_n\}$ the usual sets. We define a
partition of the element $g\in G$ with $i$ parts in $D$ as a
multiset $\{\{a_1,a_2,\ldots ,a_i\}\}$ such that all $a_k$'s are
nonzero elements in $D$
 and  $$a_1+a_2+\ldots +a_i=g.$$
Then the number of these partitions is denoted by $P_D(i,g)$, i.e.,
$$P_D(i,g)=\Big |\Big  \{\{\{a_1,a_2,\ldots ,a_i\}\}\subseteq D:a_1+a_2+\ldots
+a_i=g,  a_1, \ldots, a_i \neq 0 \Big \}\Big |.$$ It turns out $M(D,
i, g) = \sum_{k=0}^{i} P_D(k, g)$  is the number of partitions of
$g\in G$ with at most $i$ parts in $D$.  Another motivation to study
the enumeration of multisubset sums is due to a recent study of a
conjecture on polynomials of prescribed ranges over a finite field
\cite{GHNP:10,MW:12}.  Let $\mathbb{F}_q$ be a finite field of $q$
elements and $\mathbb{F}_q^*$ be the cyclic multiplicative group.
When $D = \mathbb{F}_q$ (the additive group) or $\mathbb{F}_q^{*}$,
counting the multisubset sum problem is the same as counting
partitions over finite fields \cite{MW:13}.

In this note, we use the same technique as in \cite{Kosters:13} to
obtain $M(D, i, g)$ when $D = G$.

\begin{Thm}
Let $G$ be a finite abelian group of size $n$ and let $g\in G$,
$i\in Z$ with $i \geq 0$. Then we have
\[
M(G, i, g) = \frac{1}{n} \sum_{s \mid \gcd(exp(G), i)}
\binom{n/s+i/s-1}{i/s} \sum_{ d \mid gcd(s, e(g))} \mu(s/d) \# G[d],
\]
where $exp(G)$ is the exponent of $G$, $e(g) = \max \{ d : d \mid
exp(G), g \in dG \}$, $\mu$ is the M\"{o}bius function, and $G[d] =
\{h \in G: dh =0\}$ is the $d$-torsion of $G$.
\end{Thm}

\begin{proof}

Let $\mathbb{C}$ be the field of complex numbers and $\hat{G} =
Hom(G, \mathbb{C}^*)$ be the group of characters of $G$. Let $\chi
\in \hat{G}$ and $\bar{\chi}$ be the conjugate character. Then we
can extend $\chi$ to a $\mathbb{C}$-algebra morphism $\chi :
\mathbb{C}[G] \rightarrow \mathbb{C}$ on the group ring
$\mathbb{C}[G]$.

First we have

\[
\sum_{i=0}^{\infty} \sum_{g\in G} M(G, i, g)g X^i = \prod_{\sigma
\in G} \frac{1}{1-\sigma X} \in \mathbb{C}[G][X]
\]

Use Lemma 2.1 \cite{Kosters:13}, we write

\[
\sum_{i=0}^{\infty}  M(G, i, g) X^i = \frac{1}{n} \sum_{\chi \in
\hat{G}} \bar{\chi}(g)  \prod_{\sigma \in G} \frac{1}{1-
\chi(\sigma) X}.
\]

Following the same arguments as in the proof of Theorem 1.1 in
\cite{Kosters:13}, we obtain

\[
\sum_{i=0}^{\infty}  M(G, i, g) X^i = \frac{1}{n} \sum_{s \mid
exp(G)} \sum_{d \mid gcd(s, e(g))} \mu(s/d) \# G[d]
\frac{1}{(1-X^s)^{n/s}}.
\]

We single out $M(G, i, g)$ and get
\[
M(G, i, g) = \frac{1}{n} \sum_{s \mid exp(G)} \sum_{d \mid gcd(s,
e(g))} \mu(s/d) \# G[d] \binom{n/s+i/s-1}{i/s}.
\]

Hence
\[
M(G, i, g) = \frac{1}{n} \sum_{s \mid \gcd(exp(G), i)}
\binom{n/s+i/s-1}{i/s} \sum_{ d \mid gcd(s, e(g))} \mu(s/d) \# G[d].
\]
\end{proof}

\begin{remark} If $\gcd(i,exp(G))=1$ we have
$M(G,i,g)=\frac{1}{n}\binom{n+i-1}{i}$.
\end{remark}

\begin{remark} Let $p$ be prime, $G=\mathbb{Z}_{n}$, and $n=p^m$. Consider $g=kp^u < p^m$ where $p\nmid k$ and
$e(g)=p^u$. For $i=tp^w$ where $\gcd(p,t)=1$, we have
\begin{eqnarray*}
M(G,i,g)&=&\frac{1}{n}\left[\sum_{h=0}^{min(w,m)}\binom{p^{m-h}+tp^{w-h}-1}{tp^{w-h}}\sum_{c=0}^{\min\{h,u\}}\mu
(p^{h-c})p^c\right]\\
&=& \frac{1}{n}\left[\binom{p^{m}+tp^{w}-1}{tp^{w}}+\sum_{h=1}^{min\{w,u\}}\binom{p^{m-h}+tp^{w-h}-1}{tp^{w-h}}
(p^h-p^{h-1}) \right.\\
&& \left.-A\binom{p^{m-u-1}+tp^{w-u-1}-1}{tp^{w-u-1}}p^u\right],
\end{eqnarray*}
where $A=1$ if $u <  min(w, m)$ and zero otherwise. Here we used the fact that
$$\sum_{c=0}^{\min\{h,u\}}\mu
(p^{h-c})p^c=\begin{cases}1, &{\text if } \quad h=0,\\ p^h-p^{h-1}, & {\text if}\quad h\leq u,\\
-p^u,& {\text if}\quad h=u+1,\\ 0, & {\text if}\quad
h>u+1.\end{cases}$$
\end{remark}

Similarly,
\[
\sum_{i=0}^{\infty} \sum_{g\in G} M(G\setminus \{0\}, i, g)g X^i =
\prod_{\sigma \in G, \sigma \neq 0} \frac{1}{1-\sigma X} \in
\mathbb{C}[G][X].
\]
Because $\chi(0) =1$,  we have
\[
\sum_{i=0}^{\infty}  M(G \setminus \{0\}, i, g) X^i = \frac{1}{n}
\sum_{s \mid exp(G)} \sum_{d \mid gcd(s, e(g))} \mu(s/d) \# G[d]
\frac{1-X}{(1-X^s)^{n/s}}.
\]

Therefore,
\begin{eqnarray*}
M(G \setminus \{0\}, i, g) &=& \frac{1}{n} \sum_{s \mid \gcd(exp(G),i)} \binom{n/s+i/s-1}{i/s}
\sum_{d \mid gcd(s, e(g))} \mu(s/d) \# G[d] \\
&& - \frac{1}{n} \sum_{s \mid \gcd(exp(G),i-1)}
\binom{n/s+(i-1)/s-1}{(i-1)/s}  \sum_{d \mid gcd(s, e(g))} \mu(s/d)
\# G[d].
\end{eqnarray*}

We note $M(G \setminus \{0\}, i, g) = P_G(i, g)$.  Therefore we
obtain an explicit formula for the number of partitions of $g$ into
$i$ parts over $G$.  More generallly, let $D = G\setminus S$, where
$S=\{u_1,u_2,\ldots ,u_{|S|}\} \neq \emptyset$. Denote by
$M_S(G,i,g)$  the number of multisubsets of $G$ with sizes $i$ that
contains at least one elements from $S$. Then the number of
multisubsets of $D = G\setminus S$  with $i$ parts which sum up to
$g$  is equal to
$$M(G\setminus S,i,g)=M(G,i,g)-M_S(G,i,g).$$

Denote $M(G,0,0)=1$ and $M(G,t,s)=0$ for $s\neq 0$ and $t\leq 0$.
The principe of the inclusion-exclusion immediately implies that
$M_S(G, i,g)$ is given in the following formula. We note that the
formula is in particular useful when the size of $S$ is small.
\begin{Prop}\label{Inclusion-Exclusion} For all $i=1,2,\ldots $ and $g\in G$ we have
$$M_S(G,i,g)=\sum_{u\in S}M(G,i-1,g-u)- \ldots$$
$$+ (-1)^{t-1} \sum_{\{u_1,u_2,\ldots ,u_t\}\subseteq
S}M(G,i-t,g-(u_1+u_2+\ldots +u_t)) + \ldots$$
$$+ (-1)^{i-2} \sum_{\{u_1,u_2,\ldots ,u_{i-1}\}\subseteq
S}M(G,1,g-(u_1+u_2+\ldots +u_{i-1})) +$$
$$(-1)^{i-1} \sum_{\{u_1,u_2,\ldots ,u_{i}\}\subseteq
S}M(G,1,g-(u_1+u_2+\ldots +u_{i})).$$
\end{Prop}

\begin{proof} Fix an element $g\in G$.
Denote by $\mathcal{A}_u$ the family of all the multisubsets of $G$
with $i$ parts which sum up to $g$ and each multisubset also
contains the element $u$. The principle of the inclusion-exclusion
implies that
\begin{equation}\label{IE}
|\cup_{u\in S} A_u|=\sum_{u\in
S}|\mathcal{A}_u|-\sum_{\{u_1,u_2\}\subseteq S
}|\mathcal{A}_{u_1}\cup\mathcal{A}_{u_2}|+\ldots 
+(-1)^{|S|-1}|\mathcal{A}_{u_1}\cap\ldots
\cap\mathcal{A}_{u_{|S|}}|.
\end{equation}

For each multisubset $\{\{ a_1,a_2,\ldots ,a_i\}\}\subseteq
\mathcal{A}_{u_1}\cap \ldots \mathcal{A}_{u_t}$ with $i$ parts which
sum up to $g$, we can assume that
 $a_1=u_1,\ldots,  a_t=u_t$. Then we obtain a multisubset $\{\{ a_{t+1},a_{t+2},\ldots
,a_i\}\} \subseteq G$ with $i-t$ parts which sum up to $g-(u_1+u_2+
\cdots +u_t)$.   Conversely, for each
 multiset $\{\{ a_{t+1},a_{t+2},\ldots,a_i\}\} \subseteq G$ with $i-t$ parts which sum up to $g-(u_1+u_2+\cdots +u_t)$,
we can obtain a multisubset $\{\{ a_{u_1},a_{u_2},\ldots, a_{u_t},
a_{t+1}, \ldots, ,a_i\}\}\subseteq \mathcal{A}_{u_1}\cap \ldots
\mathcal{A}_{u_t}$. Hence there is a bijective correspondence
between multisubsets in $\mathcal{A}_{u_1}\cap \ldots
\mathcal{A}_{u_t}$ and multisubsets of $G$  with $i-t$ parts which
sum up to
$g-(u_1+\cdots +u_t)$. There are $M(G,i-t,g-(u_1+\cdots +u_t))$ of them.  Obviously, $M(G,i-t,\star) = 0$ for $t > i$ and for $t \leq i$ we have 
$$\sum_{\{u_1,u_2,\ldots ,u_t\}\subseteq
S}|\mathcal{A}_{u_1}\cap\ldots
\cap\mathcal{A}_{u_t}|=\sum_{\{u_1,u_2,\ldots ,u_t\}\subseteq
S}M(G,i-t,g-(u_1+u_2+\ldots +u_t)). $$
Substituting this in the formula (\ref{IE}) gives the desired
result.
\end{proof}

For example, let $G= \mathbb{Z}_4$ and $S = \{0, 1 \}$. Then by
explicit counting, we can verify that $M_S(\mathbb{Z}_4\setminus
S,3,1)= M(\mathbb{Z}_4,3,1) -  M_S(\mathbb{Z}_4,3,1) =
M(\mathbb{Z}_4,3,1)
-\Big(M(\mathbb{Z}_4,2,1)+M(\mathbb{Z}_4,2,0)\Big)+M(\mathbb{Z}_4,1,0)=5-(2+3)+1=1$.

Similarly, let $S = \{0, 1, 2 \}$. We also have
$M_S(\mathbb{Z}_4\setminus S,3,1)= M(\mathbb{Z}_4,3,1) -
M_S(\mathbb{Z}_4,3,1) =
M(\mathbb{Z}_4,3,1)-\Big(M(\mathbb{Z}_4,2,1)+M(\mathbb{Z}_4,2,0)+M(\mathbb{Z}_4,2,3)\Big)+$$
$$\Big(M(\mathbb{Z}_4,1,0)+M(\mathbb{Z}_4,1,3)+M(\mathbb{Z}_4,1,2)\Big)-M(\mathbb{Z}_4,0,2)=
5-(2+3+2)+(1+1+1)-0=1$.

Another example is for group $G=\mathbb{Z}_u\times \mathbb{Z}_v$ and
$G\setminus S=\mathbb{Z}_u\times \{0\}$, we have that $M(G\setminus
S, i, (a,b))=0$ for all $b\in \mathbb{Z}_v\setminus\{0\}$, but
$M(G\setminus S,i,(a,0))\neq 0$.

\begin{remark}
We note that the same formula does not work for subset sum problem.
For example, $N(\mathbb{Z}_4,2,0)=1, \quad N(\mathbb{Z}_4,1,0)=1$,
but $N(\mathbb{Z}_4\setminus\{0\},2,0)=1\neq
N(\mathbb{Z}_4,2,0)-N(\mathbb{Z}_4,1,0)=1-1=0$. A similar formula
that holds for subset sum problem is
$$N(G\setminus S,i,g)=\sum_{u\in S}N(G\setminus \{u\},i-1,g-u)- \ldots$$
$$+ (-1)^{t-1} \sum_{\{u_1,u_2,\ldots ,u_t\}\subseteq
S}N(G\setminus \{u_1,u_2,\ldots ,u_t\},i-t,g-(u_1+u_2+\ldots +u_t))
+ \ldots$$
$$+ (-1)^{i-2} \sum_{\{u_1,u_2,\ldots ,u_{i-1}\}\subseteq
S}N(G\setminus \{u_1,u_2,\ldots ,u_{i-1}\},1,g-(u_1+u_2+\ldots
+u_{i-1}))+
$$
$$(-1)^{i-1} \sum_{\{u_1,u_2,\ldots ,u_{i}\}\subseteq
S}N(G\setminus \{u_1,u_2,\ldots ,u_{i}\},0,g-(u_1+u_2+\ldots
+u_{i})),$$ where we again use notation $N(D,0,0)=1$ and
$N(D,0,z)=0$ for $z\neq 0$.
\end{remark}

\end{document}